\documentclass[a4paper]{amsart}
\usepackage[margin=3.5cm]{geometry}
\usepackage[utf8]{inputenc}
\usepackage{amsmath,amsfonts,amssymb,amsthm, multirow}
\usepackage{mathtools,bbold}
\usepackage{tikz-cd}
\usepackage{comment, stackrel}
\usepackage[shortlabels]{enumitem}
\usepackage[utf8]{inputenc}
\usepackage[T1]{fontenc}
\usepackage{textcomp}
\usepackage[numbers]{natbib}
\usepackage{makecell}
\usepackage{hhline}
\usepackage{tikz}
\usepackage{tikz-cd}
\usepackage{pifont}
\usepackage{diagbox}
\usepackage{float}
\usepackage[dvipsnames]{xcolor}
\usepackage{natbib}   % For citations

\theoremstyle{plain}
\newtheorem{theorem}{Theorem}[section]

\newtheorem{proposition}[theorem]{Proposition}

\newtheorem{corollary}[theorem]{Corollary}
\newtheorem{lemma}[theorem]{Lemma}

\renewcommand{\phi}{\varphi}

\newcommand{\xmark}{\ding{55}}

\theoremstyle{definition}
\newtheorem{definition}[theorem]{Definition}

\theoremstyle{remark}
\newtheorem{remark}[theorem]{Remark}

\newcommand{\Z}{\mathbb{Z}}                    % Integers
\newcommand{\R}{\mathbb{R}}                    % Real numbers
\newcommand{\C}{\mathbb{C}}                    % Complex numbers
\renewcommand{\P}{\mathbb{P}}                    % 
\newcommand{\F}{\mathbb{F}}  
\newcommand{\G}{\mathbb{G}}                    % scroll 
\renewcommand{\O}{\mathcal{O}}                    % 

\newcommand{\VV}{\mathbb{V}}

\newcommand{\Proj}{\operatorname{Proj}}

\newcommand{\Pic}{\operatorname{Pic}}

\newcommand{\coker}{{\rm Cokernel}}

\usepackage[colorlinks=true,
linkcolor=Maroon,      % Color for \ref, \eqref
citecolor=NavyBlue,    % Color for \cite
urlcolor=BrickRed      % For URLs if any
]{hyperref}

\title{Calabi--Yau threefolds fibred in codimension two K3 surfaces}

\author{Geoffrey Mboya}
\address{Mathematical Institute,
	University of Oxford,
	United Kingdom}
\email{mboya@maths.ox.ac.uk}

\begin{document}
	\begin{abstract}
		We conduct a systematic search of codimension 2 Complete Intersection Calabi--Yau threefolds (CICY3) in rank 2 toric ambient spaces that are fibered by complete intersection of a quadric and a cubic in $\C\P^4$. We classify  both the nonsingular ones as well as those with isolated singularities.
	\end{abstract}

\maketitle
%\section{Introduction}
\noindent This paper is a continuation of a project started in \cite{GS23}. In that paper, Calabi–Yau threefolds with isolated singularities, admitting fibration in K3 hypersurfaces $$S_4\subset\C\P^3,S_5\subset\C\P(1,1,1,2),S_6\subset\C\P(1,1,1,3)\text{ and }S_6\subset\C\P(1,1,2,2)$$ and embedded as anticanonical hypersurfaces
in weighted scrolls were studied. Here, we extend the analysis by studying codimension two complete intersections with isolated singularities in weighted scrolls, fibred by K3 surfaces $S_{2,3}\subset\C\P^4;$ the complete intersection of a quadric and a cubic.\\[2mm]
\noindent One potential use of the objects we classify in this paper is in extending the construction of new Super-Conformal Field Theories (SCFT). As demonstrated in \cite{Bingyi_chen}, a zoo of SCFT were constructed from rational threefold Complete Intersection Isolated Singularities.\\[2mm]
\noindent From the database of nonsingular Complete Intersection Calabi--Yau threefolds (CICY3) in \cite{Candelas,he_cand} whose ambient spaces are products of projective spaces; those with elliptic or K3 or a chain of elliptic and K3 fibration structure have been studied in \cite{anderson_lara}. In this paper, we construct CICY3 families of $S_{2,3}$ complete intersection K3 with at most isolated singularities and whose ambient spaces are rank 2 toric varieties.\\[1mm] 
\begin{definition}[Fivefold Smooth Scroll]
	Let five integers $a_i$ have the following ordering $0=a_1\leq a_2\leq a_3\leq a_4\leq a_5.$ A smooth fivefold scroll $\F_A=\F(0,a_2,a_3,a_4,a_5)$ is a rank 2 toric variety associated to the weight matrix (for detailed definition see Section \ref{prelim}) $$A=\begin{bmatrix}0&-a_2&\ldots&-a_5&1&1\\1&1&\ldots&1&0&0\end{bmatrix}.$$
\end{definition} 
\noindent We prove the following result
\begin{theorem}
	\label{CYttci}
	Let $X=\VV(f_1,f_2),$ embedded in $\F_A=\F(0,a_2,a_3,a_4,a_5),$ be a threefold fibred over $\P^1$ by complete intersections $S_{2,3}\subset\C\P^4$ of a quadric and a cubic. Let $D_1, D_2\in\Pic(\F_A)$ with $D_1+D_2=-K_{\F_A}$ and consider general sections $$f_1\in H^0(\F_A, D_1)\cong\C[\F_A]_{\left(p,2\right)},f_2\in H^0(\F_A, D_2)\cong\C[\F_A]_{\left(2-p-\sum a_j,3\right)}.$$
	The resulting threefold $X$ is Calabi--Yau with at most isolated singularities in 12 cases along either of the base loci $B_i=Bs(|D_i|).$ Table \ref{tfoldtblCi} is a list of these families.  
		\begin{table}[h!]
		\centering
		\resizebox{\textwidth}{!}{\begin{tabular}{ |c|c|c| }
				\hline
				$\text{No}.$&$X\subset\F(0,a_2,a_3,a_4,a_5)=\F_A$& $\text{Description of a General }X=\begin{bmatrix}p&2-p-\sum a_j\\2&3\end{bmatrix}$\\ \hline
				1&$\begin{bmatrix}0&2\\2&3\end{bmatrix}\subset\F(0,0,0,0,0)$&$|D_1|,|D_2|\text{ are base-point-free }$\\ 
				&&$ \text{General }X\text{ is nonsingular }CY \text{ threefold}.$\\ \hline
				2&$\begin{bmatrix}0&1\\2&3\end{bmatrix}\subset\F(0,0,0,0,1)$&$|D_1|,|D_2|\text{ are base-point-free}$\\ 
				&&$ \text{General }X\text{ is nonsingular }CY \text{ threefold}.$\\ \hline
				3&$\begin{bmatrix}0&0\\2&3\end{bmatrix}\subset\F(0,0,0,0,2)$&$|D_1|,|D_2|\text{ are base-point-free}$\\ 
				&&$ \text{General }X\text{ is nonsingular }CY \text{ threefold}.$\\ \hline
				4&$\begin{bmatrix}0&0\\2&3\end{bmatrix}\subset\F(0,0,0,1,1)$&$|D_1|,|D_2|\text{ are base-point-free}$\\ 
				&&$ \text{General }X\text{ is nonsingular }CY \text{ threefold}.$\\ \hline
				5&$\begin{bmatrix}1&1\\2&3\end{bmatrix}\subset\F(0,0,0,0,0)$&$|D_1|,|D_2|\text{ are base-point-free}$\\ 
				&&$ \text{General }X\text{ is nonsingular }CY \text{ threefold}.$\\ \hline
				6&$\begin{bmatrix}1&0\\2&3\end{bmatrix}\subset\F(0,0,0,0,1)$&$|D_1|,|D_2|\text{ are base-point-free}$\\ 
				&&$ \text{General }X\text{ is nonsingular }CY \text{ threefold}.$\\ \hline
				7&$\begin{bmatrix}2&0\\2&3\end{bmatrix}\subset\F(0,0,0,0,0)$&$|D_1|,|D_2|\text{ are base-point-free}$\\ 
				&&$ \text{General }X\text{ is nonsingular }CY \text{ threefold}.$\\ \hline
				8&$\begin{bmatrix}-3&0\\2&3\end{bmatrix}\subset\F(0,0,1,2,2)$&$|D_1|\text{ is base-point-free},\dim Bs(|D_2|)=3$\\ 
				&&$\text{General }X\text{ is nonsingular}$\\\hline
				9&$\begin{bmatrix}-2&0\\2&3\end{bmatrix}\subset\F(0,0,0,2,2)$&$|D_1|\text{ is base-point-free},\dim Bs(|D_2|)=3$\\ 
				&&$\text{General }X\text{ has 6 isolated singular points}$\\
				&&$\text{along }Bs(|D_2|)$\\\hline
				10&$\begin{bmatrix}-1&1\\2&3\end{bmatrix}\subset\F(0,0,0,1,1)$&$|D_1|\text{ is base-point-free},\dim Bs(|D_2|)=3$\\ 
				&&$\text{General }X\text{ is nonsingular}$\\\hline
				11&$\begin{bmatrix}-1&0\\2&3\end{bmatrix}\subset\F(0,0,0,1,2)$&$|D_1|\text{ is base-point-free},\dim Bs(|D_2|)=3$\\ 
				&&$\text{General }X\text{ has 6 isolated singular points}$\\
				&&$\text{along }Bs(|D_2|)$\\\hline
				12&$\begin{bmatrix}-1&0\\2&3\end{bmatrix}\subset\F(0,0,1,1,1)$&$|D_1|\text{ is base-point-free},\dim Bs(|D_2|)=2$\\ 
				&&$\text{General }X\text{ has 2 isolated singular points}$\\
				&&$\text{ along }Bs(|D_2|)\cong\P^1\times\P^1$\\\hline
		\end{tabular}}
		\caption{The data $(p,a_j)$ for which a general codimension two Calabi--Yau threefold in $|L_{p,2}|\cap| L_{2-p-\sum a_j,3}|\subset\F$ is non singular or has isolated singularities.}
		\label{tfoldtblCi}
	\end{table}
\end{theorem}
\noindent {\bf Acknowledgements} This work is adapted from my DPhil thesis \cite{mboya} written at the University of Oxford. This work was completed in the Winter of 2022 through a generous support as a visiting research student at University of Vienna. I am grateful for the guidance and support of my supervisor, Bal\'azs Szendr\H oi, during and after my research. I have also been supported by  the Simons Foundation Grant $488625.$
\section{Preliminaries}
\label{prelim}
\noindent Take $(a_j,b_j)\in\Z^2$ for $j=1,\ldots,n$ such that  
\begin{align*}
&0\leq a_1\leq\ldots\leq a_n\text{ and }\\
&1\leq b_1\leq\ldots\leq b_n.
\end{align*} 
Consider the following $(\C^*)^2$ action on the product  $(\C^{n}\setminus\{0\})\times(\C^2\setminus\{0\})$
\begin{equation}
(\lambda, \mu) \colon (x_1,\ldots,x_n;x_{n+1},x_{n+2}) \mapsto (\lambda^{-a_1}\mu^{b_1}x_1,\ldots, \lambda^{-a_{n}}\mu^{b_{n}}x_{n};\lambda x_{n+1},\lambda x_{n+2}).
\label{action}
\end{equation}
An  $n$-dimensional rank 2 toric variety $\F_A$ associated to the integral weight matrix $$A=\begin{bmatrix}
-a_1&\ldots&-a_{n}&1&1\\b_1&\ldots&b_{n}&0&0\end{bmatrix}$$ is the quotient
\[ \F_A = ((\C^{n}\setminus\{0\})\times(\C^2\setminus\{0\}))/(\C^*)^2.
\] 
As in \cite{reid_chapters}, we also have the toric projection $$\pi:\Proj\left(\bigoplus_{i=1}^5\O_{\P^1}(a_i)^{\oplus b_i}\right)\cong\F_A\xrightarrow{[x_1,\ldots,x_5;x_{6},x_{7}]\mapsto[x_6,x_7]}\P^1\text{ with fibres }\P^4_{[x_1:\ldots:x_5]}[b_i].$$
In this paper, the weight matrix will be
$$A=\begin{bmatrix}-a_1&-a_2&\ldots&-a_5&1&1\\b_1&b_2&\ldots&b_5&0&0\end{bmatrix}=\begin{bmatrix}0&-a_2&\ldots&-a_5&1&1\\1&1&\ldots&1&0&0\end{bmatrix}.$$ 
In particular, the quotient $\F_A$ has no stabilizers for any choice of $(a_j)$ and the fibres of $\pi$ are $\P^4.$\\[2mm]
Since the action \eqref{action} gives an isomorphic quotient up to $S_{5}\times GL(2,\Z)$ action on the weight matrix $A,$ we assumed that $(a_j)=(0,a_2,\ldots,a_n).$ Further, for $(b_j)=(1,b_2,\ldots,b_n)$ we have, from \cite{GS23}, the following exact sequence of abelian groups 
\begin{equation} 0 \longrightarrow \Z^2 \xrightarrow[]{\begin{bmatrix}
	(1,0)\\
	(0,1)
	\end{bmatrix}\mapsto A}\Z^{n+2} \longrightarrow\frac{\Z[\rho_1, \ldots, \rho_n, \sigma_1, \sigma_2]}{\left(\sigma_1+\sigma_2 - \sum_{j=1}^n a_j \rho_j,~\sum_{j=1}^n b_j \rho_j\right)}=:N\longrightarrow 0 \label{defN} 
\end{equation}
for which $N$ is a free $\Z$-module. We define a fan $\Sigma\subset N_\R\cong\R^n$ by taking
\[\Sigma(1)=\{\rho_1, \ldots, \rho_n, \sigma_1, \sigma_2\},~\Sigma(n)=\{\tau_{i,j} = {\rm Span}(\rho_1, \ldots, \rho_{i-1}, \rho_{i+1} \ldots, \rho_{n},\sigma_j)\}\]
so that, from \cite{fulton}, we have $X_{N,\Sigma}\cong\F_A.$\\[2mm]
The $(\C^*)^2$-action induces a bigrading on the Cox Ring $S(\Sigma)=\C[\F_A]=\C[x_1,\ldots,x_{n};x_{n+1},x_{n+2}]$ which assigns the weights $(-a_i,b_i)$ to the Cox variable $x_i$ as described in \cite{Cox}.\\[2mm] 
From the group $Cl(\F_A)$ of Weil divisors on $\F_A$ with the torus invariant irreducible divisors $D_{\rho_i} = \{x_i=0\}\subset\F_A,$ a monomial $$\prod_{i=1}^{n+2}x_i^{q_i}\in S(\Sigma)=\C[x_1,\ldots,x_{n+2}]=\bigoplus_{(d_1,d_2)\in\Z^2}S_{(d_1,d_2)}$$ has bi-degree $${d_1\choose d_2}=\left[\sum_{i=1}^{n+2}q_iD_{\rho_i}\right]=q_1{-a_1\choose b_1}+\ldots +q_{n+2}{-a_{n+2}\choose b_{n+2}}.$$
It can also be computed that $-K_{\F_A}=\left(2-\sum a_j\right)D_{\sigma_1}+(\sum b_j)D_{\rho_1}$ so that  $$H^{0}(\F_A,-K_{\F_A})=\C[x_1,\ldots,x_5;x_6,x_7]_{\left(2-\sum_{j=1}^7 a_j,\sum_{j=1}^7 b_j\right)}=\C[x_1,\ldots,x_5;x_6,x_7]_{\left(2-\sum_{j=1}^5 a_j,5 \right)}.$$
Now, let $D_1,D_2,K_{\F_A}\in\Pic(\F_A)$ be such that $K_{X}=(D_1+D_2+K_{\F_A})|_{X}=\O_X,$  with the general sections
\begin{align*}
f_1\in H^{0}(\F_A,D_1)=&\C[x_1,\ldots,x_5;x_6,x_7]_{\left(p,2 \right)} \text{ and }\\
f_2\in H^{0}(\F_A,D_2)=&\C[x_1,\ldots,x_5;x_6,x_7]_{\left(2-p-\sum_{j=1}^5 a_j,3\right)}.
\end{align*} 
We can then define a codimension 2 threefold $$X=\VV(f_1,f_2)\in|D_1|\cap|D_2|\subset\F_A=\F(0,a_2,a_3,a_4,a_5)$$ fibred by complete intersection $S_{2,3}$ of a quadric and a cubic. We denote $X$ by  
\[
X=\begin{bmatrix}\deg{f_1}&\deg{f_2}\end{bmatrix}=\begin{bmatrix}p&2-p-\sum a_j\\2&3\end{bmatrix}.\]

\section{Construction: Threefold families of K3 $X_{2,3}\subset\P^4$ in $\F_A$}
We would like to solve the following classification problem: find the integral data $(p,a_i),$ where $0=a_1\leq a_2\leq a_3\leq a_4\leq a_5,$ for which the codimension 2 threefold 
$$X\hookrightarrow\F_A=\F(0,a_2,a_3,a_4,a_5)$$
defined by a generic choice of sections $f_1\in |L_{p,2}|$ and $f_2\in|L_{2-p-\sum a_j,3}|$ is either non singular or has isolated singularities.
\begin{lemma}
	\label{bpf2}
	Let $|D_1|$ and $|D_2|$ be base point free linear systems on $\F_A.$ The codimension two variety $X=\VV(f_1,f_2)$ with general sections $f_i$ of $|D_i|$ is nonsingular.
\end{lemma} 
\begin{proof}
	By Bertini's Theorem, since $|D_1|$ is base point free, a general section $X_1:=\VV(f_1)\in|D_1|$ is a nonsingular hypersurface in $\F_A.$ Since $|D_2|$ is base point free, the restriction $D_2|_{X_1}$ is basepoint-free with at least as many sections as $|D_2|.$  Therefore, with $f_2$ being the equation of a general section of $|D_2|,$ the general section $X=\VV(f_1,f_2)$ of $D_2|_{X_1}$ is also nonsingular.
\end{proof}
\subsection{Proof of Theorem \ref{CYttci}}
The equations of general sections
\[f_1(x_i)=\sum_{\substack{(q_j)\vdash2}}\alpha_{(q_j)}(x_6,x_7)x^{q_j}\in S_{(p,2)},~f_2(x_i)=\sum_{\substack{(q_j)\vdash3}}\beta_{(q_j)}(x_6,x_7)x^{q_j}\in S_{(2-p-\sum a_j,3)}\]
are such that 
the degrees of coefficients of monomial $x^{q}=\prod_{j=1}^5 x_j^{q_j}$ are 
\begin{align*}
\deg\alpha_{(q_j)}(x_6,x_7)&=p+\sum_{j=2}^5q_ja_j,\\
\deg\beta_{(q_j)}(x_6,x_7)&=2-p+\sum_{j=2}^5(q_j-1)a_j.
\end{align*} 

	Consider the case when  $|D_1|=|L_{p,2}|$ and $|D_2|=|L_{2-p-\sum a_j,3}|$ are both base point free on $\F$ which is implied by $2$ inequalities
	\begin{align}
	\label{ineq8b}
	p\geq0&\text{ and}\\\nonumber
	2-p-\sum a_j\geq0.&\\\nonumber
	\end{align} 
	From inequalities \eqref{ineq8b} and $a_5\geq a_4\geq a_3\geq a_2\geq0,$ we get the following data of nonsingular Calabi--Yau threefold $X\in|D_1|\cap|D_2|$ embedded in $5$-fold scrolls
	\begin{align*}
	(p,a_2,a_3,a_4,a_5)=&(0,0,0,0,0),(0,0,0,0,1),(0,0,0,0,2),(0,0,0,1,1),(1,0,0,0,0)\\
	&(1,0,0,0,1)\text{ or }(2,0,0,0,0).
	\end{align*}
	By irreducibility of $X,$ both $B_1=Bs(|L_{p,2}|), B_2=Bs(|L_{2-p-\sum a_j,3}|)\subset\F$ are each of dimension at most three. Equivalently $\alpha_{(00020)}(x_6,x_7)x_4^2\text{ appears in } f_1 \text{ and }\beta_{(00030)}(x_6,x_7)x_4^3$ appears in $f_2$
	\begin{align}
	\label{beq3ci}
	p+2a_4\geq0&\text{ and}\\\nonumber
	2-p-a_2-a_3+2a_4-a_5\geq0.&\\ \nonumber
	\end{align}
	Now, depending on whether $|D_1|$ and $|D_2|$ are base point free or have a base locus of dimension at most $3,$ there are at most $15$ more cases to consider for potential data $(p,a_2,\ldots,a_5)$ for which  $X$ has at most isolated singularities.
\begin{center}
	\begin{table}[h!]
		\begin{tabular}{|l|l|l|l|l|l|}\hline
			\diagbox[width=5.5em]{$\dim B_1$}{$\dim B_2$}& $\emptyset$& 0&  1& 2& 3 \\\hline
			$\emptyset$&\color{green}{\checkmark}7 &\color{red}\xmark  &\color{red}\xmark  &\color{red}\xmark &\color{red}\xmark \\\hline
			0&\color{green}{\checkmark}0& \color{red}\xmark &\color{red}\xmark  &\color{red}\xmark & \color{red}\xmark\\\hline
			1&\color{green}{\checkmark}0&\color{green}{\checkmark}0& \color{red}\xmark & \color{red}\xmark&\color{red}\xmark \\\hline
			2&\color{green}{\checkmark}1&\color{green}{\checkmark}0&\color{green}{\checkmark}0& \color{red}\xmark&\color{red}\xmark\\\hline
			3&\color{green}{\checkmark}4&\color{green}{\checkmark}0&\color{green}{\checkmark}0&\color{green}{\checkmark}0&\color{red}\xmark\\\hline
		\end{tabular}
		\caption{Cases to check {\color{green}{\checkmark \#}} with the number  {\color{green}{\#}} of families found in each case and cases not to check {\color{red}{\xmark}} for Potential non-singular (or isolated singularity on ) Calabi--Yau threefolds fibred in $X_{2,3}\subset\P^4$}
	\end{table}
\end{center}
	The case when $\dim B_1=\dim B_2=0$ is not possible because of the symmetry of the base variables $x_6,x_7$ and hence omitted. We also omit the cases when $$1\leq\dim(B_1)=\dim(B_2)\leq3$$ since for these cases, there is at least a whole curve of singularity along the base loci. Moreover, since $D_2$ has more sections than $D_1,$ it suffices to check cases where $\dim B_2<\dim B_1.$
	\begin{enumerate}[1.]
		\item The cases when $(\dim(B_1),\dim(B_2))=(1,0),(1,\emptyset),(2,1),(2,0),(3,2),(3,1)\text{ and }(3,0)$  do not result in Complete Intersection Calabi--Yau threefold data $(p,a_j)$.
		\item For the case when $(\dim(B_1),\dim(B_2))=(2,\emptyset),$ the Inequality \eqref{beq3ci} implies that $\alpha_{(02000)}(x_6,x_7)x_2^2$ does not appear in $f_1,$ $\alpha_{(00200)}(x_6,x_7)x_3^2$ appears in  $f_1$ and $\beta_{(3000)}(x_6,x_7)x_1^3$ appears in  $f_2.$ Equivalently
		\begin{align}
		\label{b1bpf0b1} 
		p+2a_2<0,\\\nonumber  
		p+2a_3\geq0&\text{ and}\\\nonumber           
		2-p-a_2-a_3-a_4-a_5\geq0.&
		\end{align}
		Define
		$$\Delta(x_6,x_7,x_1,x_2,0,0,0):=\bigwedge^2\begin{bmatrix}n_{11}&n_{12}&n_{13}\\n_{21}&n_{22}&n_{23}\end{bmatrix}$$ where $n_{ij}:=\frac{\partial f_i}{\partial x_{j+2}}\bigg|_{\{x_{j+2}=x_{j+3}=x_{j+4}=0\}}$
		\begin{align*}
		n_{11}&=\alpha_{(10100)}(x_6,x_7)x_1+\alpha_{(01100)}(x_6,x_7)x_2,\\
		n_{12}&=\alpha_{(10010)}(x_6,x_7)x_1+\alpha_{(01010)}(x_6,x_7)x_2,\\
		n_{13}&=\alpha_{(10001)}(x_6,x_7)x_1+\alpha_{(01001)}(x_6,x_7)x_2,\\
		n_{21}&=\beta_{(20100)}(x_6,x_7)x_1^2+\beta_{(11100)}(x_6,x_7)x_1x_2+\beta_{(20100)}(x_6,x_7)x_2^2,\\
		n_{22}&=\beta_{(20010)}(x_6,x_7)x_1^2+\beta_{(11010)}(x_6,x_7)x_1x_2+\beta_{(20010)}(x_6,x_7)x_2^2,\\
		n_{23}&=\beta_{(20001)}(x_6,x_7)x_1^2+\beta_{(11001)}(x_6,x_7)x_1x_2+\beta_{(20001)}(x_6,x_7)x_2^2.
		\end{align*}
		Note that we need
		\begin{align}
		\label{thfold2}	
		\deg(\alpha_{(01010)})=p+a_2+a_4\geq0&\text{ and}\\ \nonumber
		\deg(\beta_{(20001)})=2-p-a_2-a_3-a_4\geq0,&
		\end{align}
		for an irreducible
		\begin{align*}
		\text{Sing}(\widetilde{X})\cap\widetilde{\F(0,a_2)}&=\VV(f_1,f_2,\Delta(x_6,x_7,x_1,x_2,0,0,0),x_3,x_4,x_5)\\
		&=\VV(f_2(x_6,x_7,x_1,x_2,0,0,0),~\Delta(x_6,x_7,x_1,x_2,0,0,0),x_3,x_4,x_5).
		\end{align*}
		Hence, from Inequalities $a_5\geq a_4\geq a_3\geq a_2\geq0,$ \eqref{b1bpf0b1} and  \eqref{thfold2}  the Complete Intersection Calabi--Yau threefold
		$$X=\begin{bmatrix}
		-1&0\\
		2&3	
		\end{bmatrix}\subset\F(0,0,1,1,1)$$ with 2 isolated singular points
		\begin{align*}
		Sing(X_2)\cap\F(0,0)&=\VV(\beta_{(30000)}(x_6,x_7)x_1^3+\ldots+\beta_{(03000)}(x_6,x_7)x_2^3,~\Delta(x_6,x_7,x_1,x_2,0,0,0))\\\nonumber
		&=\VV(|L_{0,2}|_{\F(0,0)}|,|L_{1,3}|_{\F(0,0)}|)\subset\P^1\times\P^1.
		\end{align*}
		\item Consider the case when $(\dim(B_1),\dim(B_2))=(3,\emptyset),$ the Inequality \eqref{beq3ci} implies that $\alpha_{(00200)}(x_6,x_7)x_3^2$ does not appear in $f_1,\alpha_{(00020)}(x_6,x_7)x_4^2\text{ appears in }f_1$ and $\beta_{(3000)}(x_6,x_7)x_1^3$ appears in $f_2.$ Equivalently
		\begin{align}
		\label{b1bpf0b} 
		p+2a_3<0,\\\nonumber  
		p+2a_4\geq0&\text{ and}\\\nonumber           
		2-p-a_2-a_3-a_4-a_5\geq0.&\\\nonumber
		\end{align}
		Define $Sing(X)=\VV(f_1,f_2,\Delta(x_i))$ with
		$$\Delta(x_i)=\left(\text{rank}\begin{bmatrix}
		\nabla f_1\\\nabla f_2
		\end{bmatrix}<2\right)=\left(\bigwedge^2\begin{bmatrix}
		\frac{\partial f_1}{\partial x_6}&\frac{\partial f_1}{\partial x_7}&\frac{\partial f_1}{\partial x_1}&\frac{\partial f_1}{\partial x_2}&\frac{\partial f_1}{\partial x_3}&\frac{\partial f_1}{\partial x_4}&\frac{\partial f_1}{\partial x_5}\\
		\frac{\partial f_2}{\partial x_6}&\frac{\partial f_2}{\partial x_7}&\frac{\partial f_2}{\partial x_1}&\frac{\partial f_2}{\partial x_2}&\frac{\partial f_2}{\partial x_3}&\frac{\partial f_2}{\partial x_4}&\frac{\partial f_2}{\partial x_5}\end{bmatrix}\right).$$
		Now, along $\{x_4=x_5=0\}$, the threefold $X$ has singular locus 
		\begin{align*}
		\text{Sing}(\widetilde{X})\cap\widetilde{\F(0,a_2,a_3)}&=\VV(f_1,f_2,\Delta(x_6,x_7,x_1,x_2,x_3,0,0),x_4,x_5)\\
		&=\VV(f_2(x_6,x_7,x_1,x_2,x_3,0,0),~\Delta(x_6,x_7,x_1,x_2,x_3,0,0),x_4,x_5)
		\end{align*}
		where $$\widetilde{\F(0,a_2,a_3)}=q^{-1}(\VV(x_4,x_5))\subset\C^5_{x_j}\setminus\{0\}\times\C^2_{x_6,x_7}\setminus\{0\}\xrightarrow{q}\F.$$
		The locus $\text{Sing}(\widetilde{X})\cap\widetilde{\F(0,a_2,a_3)}$ are isolated singularities if it is defined by $$f_2|_{\{x_4=x_5=0\}}\not\equiv0$$ and an irreducible non-constant quadric $$\Delta(x_6,x_7,x_1,x_2,x_3,0,0)=m_{11}m_{22}-m_{12}m_{21}\not\equiv0,\text{ where }m_{ij}:=\frac{\partial f_i}{\partial x_{j+3}}\bigg|_{\{x_{j+3}=x_{j+4}=0\}}$$
		\begin{align*}
		m_{11}=&\alpha_{(10010)}(x_6,x_7)x_1+\alpha_{(01010)}(x_6,x_7)x_2+\alpha_{(00110)}(x_6,x_7)x_3,\\
		m_{12}=&\alpha_{(10001)}(x_6,x_7)x_1+\alpha_{(01001)}(x_6,x_7)x_2+\alpha_{(00101)}(x_6,x_7)x_3,\\
		m_{21}=&\beta_{(20010)}(x_6,x_7)x_1^2+\beta_{(11010)}(x_6,x_7)x_1x_2+\beta_{(10110)}(x_6,x_7)x_1x_3+\\
		&\beta_{(02010)}(x_6,x_7)x_2^2+\beta_{(01110)}(x_6,x_7)x_2x_3+\beta_{(00210)}(x_6,x_7)x_3^2,\\
		m_{22}=&\beta_{(20001)}(x_6,x_7)x_1^2+\beta_{(11001)}(x_6,x_7)x_1x_2+\beta_{(10101)}(x_6,x_7)x_1x_3+\\
		&\beta_{(02001)}(x_6,x_7)x_2^2+\beta_{(01101)}(x_6,x_7)x_2x_3+\beta_{(00201)}(x_6,x_7)x_3^2.
		\end{align*}
		Equivalently
		\begin{align}
		\label{bineq41cia}	
		\deg(\alpha_{(00110)})=p+a_3+a_4\geq0&\text{ and}\\ \nonumber
		\deg(\beta_{(00210)})=2-p-a_2+a_3-a_5\geq0.&
		\end{align}
		We then get, from Inequalities $a_5\geq a_4\geq a_3\geq a_2\geq0,$ \eqref{bineq41cia} and  \eqref{b1bpf0b},  the following data on Complete Intersection Calabi--Yau threefold in 5-fold scrolls
		$$
		(p,a_2,a_3,a_4,a_5)=(-3,0,1,2,2),(-2,0,0,2,2),(-1,0,0,1,1)\text{ and }(-1,0,0,1,2).
		$$
		We check for isolated singularities in each case above: 
		\begin{enumerate}[i)]
		\item $X_1=\begin{bmatrix}
			-3&0\\
			2&3	
		\end{bmatrix}\subset\F(0,0,1,2,2)$ with
		\begin{align*}
		Sing(X_1)\cap\F(0,0,1)&=\VV(\beta_{(30000)}(x_6,x_7)x_1^3+\ldots+\beta_{(00300)}(x_6,x_7)x_3^3,~\Delta(x_6,x_7,x_1,x_2,x_3,0,0))\\\nonumber
		&=\VV(|L_{0,3}|_{\F(0,0,1)}|,|L_{4,3}|_{\F(0,0,1)}|)=\emptyset\subset\F(0,0,1).
		\end{align*} 
		\item $X_2=\begin{bmatrix}
			-2&0\\
			2&3	
		\end{bmatrix}\subset\F(0,0,0,2,2)$ with
		\begin{align*}
		Sing(X_2)\cap\F(0,0,0)&=\VV(\beta_{(30000)}(x_6,x_7)x_1^3+\ldots+\beta_{(00300)}(x_6,x_7)x_3^3,~\Delta(x_6,x_7,x_1,x_2,x_3,0,0))\\\nonumber
		&=\VV(q_3(x_1,x_2,x_3))\cap|L_{2,3}|_{\F(0,0,0)}|\subset\P^1_{[x_6:x_7]}\times\P^2_{[x_1:x_2:x_3]}
		\end{align*}
		where $q_3=\beta_{(30000)}(x_6,x_7)x_1^3+\ldots+\beta_{(00300)}(x_6,x_7)x_3^3.$\\[2mm]
		Here we are intersecting a general $(0,3)$- and a general $(2,3)$-form on $\F(0,0,0)\cong\P^1\times\P^2.$ Now, with at most 2 vertical rulings and at most 3 horizontal rulings on $\P^1\times\P^2$, we conclude that there are $6=2\times3$ isolated singular points  $\{[\gamma_{1i}:\gamma_{2i};k_{1i}:k_{2i}:k_{3i}:0:0]:i=1,2,3\}$ on $X.$ 
		
		\item $X_3=\begin{bmatrix}
			-1&1\\
			2&3	
		\end{bmatrix}\subset\F(0,0,0,1,1)$ with
		\begin{align*}
		Sing(X_3)\cap\F(0,0,0)&=\VV(\beta_{(30000)}(x_6,x_7)x_1^3+\ldots+\beta_{(00300)}(x_6,x_7)x_3^3,~\Delta(x_6,x_7,x_1,x_2,x_3,0,0))\\\nonumber
		&=\VV(|L_{1,3}|_{\F(0,0,0)}|, |L_{2,3}|_{\F(0,0,0)})=\emptyset\subset\P^1\times\P^2.
		\end{align*}
		\item $X_4=\begin{bmatrix}
			-1&0\\
			2&3	
		\end{bmatrix}\subset\F(0,0,0,1,2)$ with 6 isolated singular points
		\begin{align*}
		Sing(X_2)\cap\F(0,0,0)&=\VV(\beta_{(30000)}(x_6,x_7)x_1^3+\ldots+\beta_{(00300)}(x_6,x_7)x_3^3,~\Delta(x_6,x_7,x_1,x_2,x_3,0,0))\\\nonumber
		&=\VV(r_3(x_1,x_2,x_3),|L_{2,3}|_{\F(0,0,0)}|)\subset\P^1\times\P^2
		\end{align*}
		where $r_3=\beta_{(30000)}(x_6,x_7)x_1^3+\ldots+\beta_{(00300)}(x_6,x_7)x_3^3.$\\
		\end{enumerate}
		\end{enumerate}

\section{Cohomology of Codimension Two CICY3 in $\F_A$}
\noindent In this section we demonstrate the computation of the cohomology for the first $7$ families in terms of the data $(p,a_j)$ using the Cayley trick in \cite[Section 2]{anvar}.\\[2mm]
Denote by $i:X\hookrightarrow\F^5(a_j)$ the inclusion for the first $7$ threefold families  in the Table \ref{tfoldtblCi}. For the cohomology in the middle dimension $3,$ we use \cite[Definition 10.9]{baty_bori} adapted to this case.
\begin{definition}
The \emph{variable cohomology} of $X$ is defined to be the image of the restriction map
\[
H^3_{\text{var}}(X) := \coker\left( H^3(\F_A, \C) \xrightarrow{i^*} H^3(X, \mathbb{C}) \right).
\]
This inherits a natural mixed Hodge structure from $H^3(X, \C)$, making it a Hodge substructure. In particular, the Hodge decomposition
\[
H^3(X, \C) = \bigoplus_{p+q=3} H^{p,q}(X)
\]
induces a decomposition
\[
H^3_{\text{var}}(X) = \bigoplus_{p+q=3} H^{p,q}(X)_{\text{var}},
\]
where $H^{p,q}(X)_{\text{var}} := H^{p,q}(X) \cap H^3_{\text{var}}(X)$.
\end{definition}

\begin{theorem}[Lefschetz Hyperplane Theorem {\cite[Chapter 6, Theorem 6.1]{VoisinHodge}}]
	Let $Y$ be a smooth projective variety, and let $X \subset Y$ be a smooth complete intersection of codimension $c$ defined by ample divisors. Then the restriction map
	\[
	H^k(Y, \mathbb{C}) \xrightarrow{\sim} H^k(X, \mathbb{C})
	\]
	is an isomorphism for all $k < \dim X$.
\end{theorem}

\begin{remark}
	In our setting, $X \subset \F_A = \F(0, a_2, a_3, a_4, a_5)$ is a Calabi–Yau threefold given by a smooth, transverse complete intersection of two divisors $D_1, D_2 \in \Pic(\F_A)$ satisfying $D_1 + D_2 = -K_{\F_A}$. If $D_1$ and $D_2$ are basepoint-free or ample, which is the case for the first $7$ threefold families  in the Table \ref{tfoldtblCi}, then the restriction map
	\[
	H^2(\F)_A, \mathbb{C}) \xrightarrow{i^*} H^2(X, \C)
	\]
	is an isomorphism by the Lefschetz theorem.
	
	\vspace{0.5em}
	\noindent \textbf{Caution:} In toric varieties such as $\F_A$, the classical Lefschetz theorem may fail if the divisors are not ample. However, when $D_1$ and $D_2$ are ample (or nef and big), a toric analogue holds (see \cite{Oda}, \cite{Danilov}, \cite{BatyrevCox}). This justifies the equality $H^{1,1}(X) = H^{1,1}(\F_A)$ and ensures that the variable part of the Hodge structure begins only in degree 3.
\end{remark}

\begin{definition}[Rank 3 Toric variety]
Let $\F_A=\F(a_1,\ldots,a_5)$ be a fivefold smooth scroll and consider general sections $$f_1\in\C[\F_A]_{\left(p,2\right)},f_2\in\C[\F_A]_{\left(2-p-\sum a_j,3\right)}.$$ A rank 3 toric variety $\G_B$ is defined by
\[ \G_B := \left((\C^{5}_{x_1,\ldots,x_5}\setminus\{0\})\times(\C^2_{x_6,x_7}\setminus\{0\})\times(\C^2_{y_1,y_2}\setminus\{0\})\right)/(\C^*)^3,\] with $$B=\begin{bmatrix}A&-\deg{f_1}&-\deg{f_2}\\0&1&1\end{bmatrix}=\begin{bmatrix}0&-a_2&\ldots&-a_5&1&1&-p&-2+p+\sum a_j\\1&1&\ldots&1&0&0&-2&-3\\0&0&\ldots&0&0&0&1&1
\end{bmatrix}$$
where $(\C^*)^3$ acts by  $$(\lambda,\mu,\nu).(x_j,x_6,x_7,y_1,y_2)=(\lambda^{-a_j}\mu^{b_j}x_j,\lambda x_{6},\lambda x_{7},\mu^{-p}\nu y_1,\mu^{-(2-p-\sum a_j)}\nu y_2).$$
The Cox ring $$S=\C[\G_B]=\bigoplus_{(d_1,d_2,d_3)\in\Z^3} S_{(d_1,d_2,d_3)}$$ is such that the trigraded pieces $S_{(d_1,d_2,d_3)}$ have three generators.
\end{definition}

\begin{proposition}
Let $X=\VV(f_1)\cap\VV(f_2)\subset\F_A=\F(a_j)$ where $\deg(f_1)=(p,2)$ and $\deg(f_2)=(2-p-\sum a_j,3)$ as in the first $7$ entries of Table \ref{CYttci}. 
Taking $$F(x_j,x_6,x_7,y_1,y_2):=y_1f_1+y_2f_2\in\C[G_B]_{(0,0,1)}$$ and $Y=\VV(F)\subset\G_B,$ let $$J(F)=\left(\partial F/\partial x_j,\partial F/\partial x_6,\partial F/\partial x_7,\partial F/\partial y_1,\partial F/\partial y_2\right)\triangleleft\C[G_B].$$ The Jacobian ring $R(F)=\C[G_B]/J(F)$ has the induced trigrading from $\C[\G_B]$ and there is a canonical isomorphism 
\[R(F)_{(0,0,9-q)}\cong H^{q-2,5-q}_{\text{var}}(X), \text{ where }q\neq3. \]
\end{proposition}
\begin{proposition}
	Let $X = \mathbb{V}(f_1) \cap \mathbb{V}(f_2) \subset \mathbb{F}_A = \mathbb{F}(a_j)$ be as in the first $7$ entries of Table~\ref{CYttci}, where $\deg(f_1) = (p,2)$ and $\deg(f_2) = (2 - p - \sum a_j, 3)$. Define
	\[
	F(x_j, x_6, x_7, y_1, y_2) := y_1 f_1 + y_2 f_2 \in \mathbb{C}[\mathbb{G}_B]_{(0,0,1)},
	\]
	and let $Y = \mathbb{V}(F) \subset \mathbb{G}_B$, with Jacobian ideal
	\[
	J(F) := \left( \frac{\partial F}{\partial x_j}, \frac{\partial F}{\partial x_6}, \frac{\partial F}{\partial x_7}, \frac{\partial F}{\partial y_1}, \frac{\partial F}{\partial y_2} \right) \triangleleft \mathbb{C}[\mathbb{G}_B].
	\]
	Then the Jacobian ring $R(F) = \mathbb{C}[\mathbb{G}_B]/J(F)$ inherits a natural trigrading from $\mathbb{C}[\mathbb{G}_B]$, and there is a canonical isomorphism
	\[
	R(F)_{(0,0,9-q)} \cong H^{q-2,5-q}_{\mathrm{var}}(X), \quad \text{for } q \neq 3.
	\]
\end{proposition}

\begin{proof}
	The result follows from the \emph{Cayley trick} in codimension $s = 2$ and ambient dimension $d = 5$. One considers the rank $2$ vector bundle $\mathcal{E} = L_{\deg f_1} \oplus L_{\deg f_2}$ over $\mathbb{F}_A$, whose fibre is coordinatized by $(y_1, y_2)$, and then constructs the projectivization
	\[
	\mathbb{G}_B = \mathbb{P}(\mathcal{E}) = \Proj(L_{(p,2)} \oplus L_{(2 - p - \sum a_j,3)})
	\]
	as a $\mathbb{P}^1$-bundle over $\mathbb{F}_A$. The hypersurface $Y = \mathbb{V}(F)$ is then a section of $\mathcal{O}_{\mathbb{P}(\mathcal{E})}(1)$, and there is a natural diagram of varieties:
	\[
	\begin{tikzcd}
		Y \arrow[rr, bend left, "\varphi"] \arrow[r, hook, "j"] & \mathbb{G}_B \arrow[r, "p"] & \mathbb{F}_A & X \arrow[l, hook', "i"]
	\end{tikzcd}
	\]
	where $p$ is the bundle projection from the projective bundle $\mathbb{G}_B$ to the projective bundle $\mathbb{F}_A$, and $i$ and $j$ are the natural closed immersions of $X$ and $Y$ into $\mathbb{F}_A$ and $\mathbb{G}_B$, respectively. The composition $\varphi := p \circ j$ maps the hypersurface $Y$ onto the base $X$.\\[2mm]
	Since $X$ is the complete intersection of global sections of basepoint-free line bundles and is smooth by assumption, $Y$ is also smooth by Lemma \ref{bpf2}. The grading on the Jacobian ring is particularly simple in this Calabi--Yau case: the defining equation has degree
	\[
	\beta = \deg F = \deg f_j + \deg y_j = (0,0,1),
	\]
	while the total coordinate degree is
	\[
	\beta_0 = \deg(x_1 \cdots x_5 x_6 x_7 y_1 y_2) = (0,0,2).
	\]
	Then, by \cite[Theorem~3.6]{anvar}, the variable part of the Hodge cohomology of $X$ is computed by the graded component:
	\[
	R(F)_{(5+2 - q)\beta - \beta_0} = R(F)_{(0,0,9 - q)} \cong H^{q-2,5-q}_{\mathrm{var}}(X), \quad \text{for } q \neq 3.
	\]
\end{proof}

\begin{corollary}
As in \cite[Theorem 4.6]{anvar}, we use the ideal $$J_1(F)=\left(x_j\partial F/\partial x_j,x_6\partial F/\partial x_6,x_7\partial F/\partial x_7,y_1\partial F/\partial y_1,y_2\partial F/\partial y_2\right)\subset\C[G_B]_{(0,0,1)}$$ whose generators are of the same tridegree $(0,0,1)$ instead of $J(F)$ to get $$R_1(F)=\C[\G_B]/J_1(F)\cong \C[\G_B]/J(F)=R(F).$$ Therefore, for $q\neq3,$ we have 
\begin{align*}
&h^{q-2,5-q}(X)=\dim\left(R_1(F)_{(0,0,9-q)}\right)=\dim \left(S_{(0,0,9-q)}/J_1(F)\cap S_{(0,0,9-q)}\right).
\end{align*}
In particular, the only nontrivial contributions in the Hodge structure on \( H^3(X, \C) \) occur in degrees:
\[	
R(F)_{(0,0,4)} \cong H^{3,0}_{\emph{var}}(X)=\C, \quad 
R(F)_{(0,0,5)} \cong H^{2,1}_{\emph{var}}(X), \quad 
R(F)_{(0,0,7)} \cong H^{0,3}_{\emph{var}}(X)=\C,
\]
with the remaining Hodge numbers determined by Serre duality  Poincare duality.

\end{corollary}	

\subsection*{Implementation in \texttt{Macaulay2}}

To determine the Hodge number
\[
h^{2,1}(X) = \dim R(F)_{(0,0,5)},
\]
we compute the dimension of the graded component of the Jacobian ring associated to the Cayley polynomial
\(F = y_1 f_1 + y_2 f_2\).  
This calculation can be carried out in the computer algebra system \texttt{Macaulay2}\cite{M2}, which conveniently supports multigraded rings and their homogeneous parts.\\[2mm]

We begin by defining the Cox ring of the ambient toric variety.  
For the projective bundle
\[
p : \mathbb{G}_B = \mathbb{P}\!\left(\mathcal{O}(p,2) \oplus \mathcal{O}(2-p-\sum a_j,3)\right) \longrightarrow \mathbb{F}_A,
\]
with parameters \((p,\sum a_j)=(0,2)\) and \(\mathbb{F}_A = \mathbb{F}(0,0,0,1,1)\), the Cox ring is
\[
S = \mathbb{C}[x_1, \dots, x_7, y_1, y_2],
\]
endowed with the trigrading
\[
\deg(x_1, \dots, x_5) = (0,0,1), \qquad
\deg(x_6, x_7) = (1,0,0), \qquad
\deg(y_1) = (0,0,1), \qquad
\deg(y_2) = (2,0,1).
\]
In \texttt{Macaulay2}, this is set up as
\[
\texttt{T = QQ[x1..x7,y1,y2, Degrees => \{\{0,1,0\},\ldots,\{2,0,1\}\}, DegreeRank => 3];}.
\]

Next, we introduce two general homogeneous polynomials
\[
f_1 \in S_{(0,2,0)}, \qquad f_2 \in S_{(0,3,0)},
\]
and form their Cayley combination
\[
F = y_1 f_1 + y_2 f_2.
\]
In \texttt{Macaulay2} these can be generated symbolically or at random, for instance:
\[
\texttt{f1 = random(QQ)*basis(\{0,2,0\},T); \quad f2 = random(QQ)*basis(\{0,3,0\},T); \quad F = y1*f1 + y2*f2.}
\]

The corresponding Euler-type Jacobian ideal is
\[
J_1(F) = \bigl( x_i \tfrac{\partial F}{\partial x_i},\,
x_6 \tfrac{\partial F}{\partial x_6},\,
x_7 \tfrac{\partial F}{\partial x_7},\,
y_1 \tfrac{\partial F}{\partial y_1},\,
y_2 \tfrac{\partial F}{\partial y_2} \bigr),
\]
implemented in \texttt{Macaulay2} as
\[
\texttt{J1F = ideal apply(\{x1..x7,y1,y2\}, z -> z*diff(F,z));}.
\]

\noindent
To extract the component of tridegree \((0,0,5)\), we use
\[
S_{(0,0,5)} = \texttt{basis(\{0,0,5\}, T)}, \qquad
(J_1(F))_{(0,0,5)} = \texttt{basis(\{0,0,5\}, J1F)},
\]
and record their dimensions:
\[
\texttt{dimT005 = numColumns basis(\{0,0,5\}, T); \quad
	dimJ005 = numColumns basis(\{0,0,5\}, J1F);}.
\]
The resulting Hodge number is obtained as
\[
h^{2,1}(X) = \dim S_{(0,0,5)} - \dim (J_1(F))_{(0,0,5)},
\]
or equivalently,
\[
\texttt{h21 = dimT005 - dimJ005;}.
\]

\noindent
Further refinements may include syzygies or higher Koszul relations reflecting the toric structure and the resolution of \(\mathcal{O}_X\).  
However, even this basic Jacobian computation is often demanding: the graded piece \(S_{(0,0,5)}\) usually has large dimension, and the corresponding subspaces in \(J_1(F)\) may involve hundreds or thousands of basis elements.  
Consequently, the linear algebra needed to obtain \(\dim R(F)_{(0,0,5)}\) can be heavy even for powerful computer algebra systems such as \texttt{Macaulay2}.

\begin{remark} We can construct other CICY3 families fibered by Complete Intersection K3 surfaces using the following a similar set up. 
\noindent The fibre $S_{2,3}\subset\C\P^4$ used in this paper is the first in the list of $84$ codimension two K3 surfaces $$S_{(d_1,d_2)}\subset\P[b_1,\ldots,b_5]$$ from Fletcher's List \cite[Section 13.8, Table 2]{fletcher}.\\[2mm]
Through an improvement of the computer algebra machinery in \cite[Appendix A]{mullet}, we can find other 83 finite Lists of codimension two CICY3 families similar to those in Table \ref{tfoldtblCi}. In each one of these cases, the set up would be as follows:\\[2mm]
 With integers  $0=a_1\leq a_2\leq a_3\leq a_4\leq a_5,$ we have that the threefold  $$X:=\begin{bmatrix}p&2-p-\sum a_j\\d_1&d_2\end{bmatrix}\hookrightarrow\F_A, A=\begin{bmatrix}0&-a_2&\ldots&-a_5&1&1\\b_1&b_2&\ldots&b_2&0&0\end{bmatrix}.$$
 We take $D_1, D_2, K_{F_A}\in\Pic(\F_A)$ and define $X=\VV(f_1,f_2)$ where $$f_1\in H^0(\F_A, D_1)\cong\C[\F_A]_{\left(p,d_1\right)},f_2\in H^0(\F_A, D_2)\cong\C[\F_A]_{\left(2-p-\sum a_j,d_2\right)}$$ with $K_{X}=(D_1+D_2+K_{\F_A})|=\O_{X}.$
For singular $X,$ we conjecture that the singularities would live on either of the base loci $B_i=Bs(D_i)$ whenever $$1\leq\dim B_2\leq\dim B_1\leq3.$$\\[2mm] 
A stronger question to the one answered in this paper would be: If you have a CICY3 $X$ fibred  over $\P^1$ in codimension two K3 surface $S,$ can we always find weight matrix $A$ for which the embedding  $X\hookrightarrow\F_A$ exists?  
\end{remark}


\begin{thebibliography}{99}
	
	\bibitem{fletcher}
	A.\,R. Iano-Fletcher,
	{\em Working with weighted complete intersections},
	in: \textit{Explicit Birational Geometry of 3-Folds}, pp.\,101--173,
	London Math. Soc. Lecture Note Ser., 281,
	Cambridge Univ. Press, Cambridge, 2000.
	
	\bibitem{anvar}
	A. Mavlyutov,
	{\em Cohomology of complete intersections in toric varieties},
	Pacific J. Math. \textbf{191} (1999), no.~1, 133--144.
	
	\bibitem{anderson_lara}
	A. Lara, G. Xin, G. James, and L. Seung-Joo,
	{\em Fibrations in CICY threefolds},
	J. High Energy Phys. (2017), no.~10, 077.
	
	\bibitem{baty_bori}
	B. Victor and B. Lev,
	{\em On Calabi--Yau complete intersections in toric varieties},
	in: \textit{Higher-Dimensional Complex Varieties} (Trento, 1994),
	pp.\,39--65, Walter de Gruyter \& Co., Berlin, 1996.
	
	\bibitem{Bingyi_chen}
	B. Chen, D. Xie, S.-T. Yau, and H. Zuo,
	{\em 4d $\mathcal{N}=2$ SCFT and singularity theory, Part~II: complete intersection},
	Adv. Theor. Math. Phys. \textbf{21} (2017), no.~1, 121--145.
	
	\bibitem{Cox}
	D.\,A. Cox,
	{\em The homogeneous coordinate ring of a toric variety},
	J. Algebraic Geom. \textbf{4} (1995), 17--50.
	
	\bibitem{mboya}
	G. Mboya,
	{\em Projective Fibrations in Weighted Scrolls},
	DPhil thesis, University of Oxford, 2023.
	
	\bibitem{GS23}
	G. Mboya and B. Szendrői,
	{\em On K3 fibred Calabi--Yau threefolds in Weighted Scrolls},
	Rend. Circ. Mat. Palermo \textbf{2} (2023).
	Available at \href{https://doi.org/10.1007/s12215-023-00933-0}{arXiv:2210.10559}.
	
	\bibitem{he_cand}
	H. Anming and P. Candelas,
	{\em On the number of complete intersection Calabi--Yau manifolds},
	Comm. Math. Phys. \textbf{135} (1990), no.~1, 193--199.
	
	\bibitem{mullet}
	J.\,P. Mullet,
	{\em Toric Calabi--Yau hypersurfaces fibered by weighted K3 hypersurfaces},
	Comm. Anal. Geom. \textbf{17} (2009), 107--138.
	
	\bibitem{dCM}
	M.\,A. de Cataldo and L. Migliorini,
	{\em The Hard Lefschetz Theorem and the topology of semismall maps},
	Ann. Sci. Éc. Norm. Supér. (4) \textbf{35} (2002), 759--772.
	
	\bibitem{reid_chapters}
	M. Reid,
	{\em Chapters on algebraic surfaces},
	in: \textit{Complex Algebraic Geometry} (Park City, UT, 1993),
	pp.\,3--159, IAS/Park City Math. Ser., 3,
	Amer. Math. Soc., Providence, RI, 1997.
	
	\bibitem{Candelas}
	P. Candelas, A.\,M. Dale, C.\,A. L\"utken, and R. Schimmrigk,
	{\em Complete Intersection Calabi--Yau Manifolds},
	Nucl. Phys. B \textbf{298} (1988), 493.
	
	\bibitem{fulton}
	W. Fulton,
	{\em Introduction to Toric Varieties},
	Annals of Mathematics Studies, vol.~131,
	Princeton Univ. Press, Princeton, NJ, 1993.
	
	\bibitem{VoisinHodge}
	C. Voisin,
	{\em Hodge Theory and Complex Algebraic Geometry~I},
	Cambridge Studies in Advanced Mathematics, vol.~76,
	Cambridge Univ. Press, Cambridge, 2002.
	
	\bibitem{Oda}
	T. Oda,
	{\em Convex Bodies and Algebraic Geometry: An Introduction to the Theory of Toric Varieties},
	Ergebnisse der Mathematik und ihrer Grenzgebiete,
	Springer-Verlag, Berlin, 1988.
	
	\bibitem{Danilov}
	V.\,I. Danilov,
	{\em The Geometry of Toric Varieties},
	Russ. Math. Surveys \textbf{33} (1978), no.~2, 97--154.
	
	\bibitem{BatyrevCox}
	V.\,V. Batyrev and D.\,A. Cox,
	{\em On the Hodge structure of projective hypersurfaces in toric varieties},
	Duke Math. J. \textbf{75} (1994), no.~2, 293--338.
	
	\bibitem{M2}
	D.\,R. Grayson and M.\,E. Stillman,
	{\em Macaulay2, a software system for research in algebraic geometry},
	Available at \url{http://www2.macaulay2.com}.
	
\end{thebibliography}
\end{document}